\documentclass[a4paper]{cas-sc}

\usepackage[authoryear,longnamesfirst]{natbib}

\def\tsc#1{\csdef{#1}{\textsc{\lowercase{#1}}\xspace}}
\tsc{WGM}
\tsc{QE}

\newproof{pf}{Proof}
\newproof{pot}{Proof of Theorem \ref{thm}}

\newtheorem{theorem}{Theorem}
\newtheorem{proposition}{Proposition}
\newtheorem{corollary}{Corollary}
\newtheorem{lemma}{Lemma}

\newtheorem{example}{Example}

\newtheorem{remark}{Remark}
\DeclareMathOperator{\rank}{rank}
\DeclareMathOperator{\nullity}{nullity}
\begin{document}
\let\WriteBookmarks\relax
\def\floatpagepagefraction{1}
\def\textpagefraction{.001}

\shorttitle{}    

\shortauthors{}  

\title [mode = title]{Arithmetic Conditions for the Invertibility of Circulant Matrices over Finite Fields}  

\tnotemark[1] 


%

\author[1]{Jeong-Ok Choi}

\cormark[1]

\fnmark[1]

\ead{jchoi351@gist.ac.kr}

\ead[url]{}

\credit{}

\affiliation[1]{organization={Department of Mathematical Sciences, Gwangju Institute of Science and Technology},
            city={Gwangju},
            postcode={61005}, 
            country={South Korea}}

\author[2]{Youngmi Hur}

\fnmark[2]

\ead{yhur@yonsei.ac.kr}

\ead[url]{}

\credit{}

\affiliation[2]{organization={Department of Mathematics, Yonsei University},
            city={Seoul},
            postcode={03722}, 
            country={South Korea}}

\cortext[1]{Corresponding author}



\begin{abstract}
This paper investigates arithmetic conditions, over finite fields, for the invertibility of circulant matrices whose first row is \((a_0,\ldots,a_{n-1})\), in terms of linear constraints on the entries \(a_0,\ldots,a_{n-1}\). The coefficients of such linear constraints are determined by field traces arising from Frobenius orbits of roots of unity, and under additional assumptions they are given by Ramanujan sums reduced modulo the characteristic of the given field. These results provide finite-field analogues of the corresponding criteria over \(\mathbb{Q}\), where the relevant Ramanujan-sum linear constraints can be formulated without such additional assumptions, while also showing some differences caused by reduction modulo the characteristic. The constraints can often determine invertibility using aggregate information about the entries of a row, rather than the particular pattern of the entries. We also discuss the polynomial gcd criterion, which gives necessary and sufficient conditions for invertibility and yields further consequences for special families of circulant matrices.

\end{abstract}





\begin{keywords}
Finite fields\sep Ramanujan sums\sep Circulant matrices\sep
\end{keywords}

\maketitle

\section{Introduction}\label{sec:intro}
This paper investigates arithmetic conditions for the invertibility of circulant matrices over finite fields. Let
\[
C=C(a_0,a_1,\ldots,a_{n-1})
\]
be an $n\times n$ circulant matrix, where $(a_0,a_1,\ldots,a_{n-1})$ is its first row. The invertibility of circulant matrices has been studied from various perspectives in the literature \cite{kra12}. Our starting point is an invertibility criterion over $\mathbb Q$ obtained in \cite{LMA}, which is expressed in terms of Ramanujan-sum linear combinations of the first-row entries. More precisely, the relevant conditions have the form
\[
\sum_{j=0}^{n-1} R_d(j)a_j\ne 0
\]
where $R_d(j)$ denotes the Ramanujan sum, for divisors $d$ of $n$.

At first sight, one might expect that the criterion over $\mathbb Q$ should transfer directly to finite fields by reducing the Ramanujan-sum conditions modulo the characteristic. However, this expectation is misleading. One of the main reasons why the proof of the result obtained in \cite{LMA} cannot be carried over directly is that cyclotomic polynomials are not necessarily irreducible over finite fields. In fact, the factorization of cyclotomic polynomials over finite fields is a nontrivial problem \cite{meyn, wang}. Rather than working directly with cyclotomic polynomials, we investigate the relevant algebraic structure over finite fields through field traces. Assuming that the characteristic of the field $F$ is relatively prime to the matrix size $n$, as required for the trace construction used here, we obtain trace-based linear constraints without requiring any further conditions. Under the additional assumption that $\mathrm{ord}_d(\mathrm{char}(F))=\phi(d)$ for every divisor $d$ of $n$, these constraints coincide with the corresponding Ramanujan-sum conditions, yielding a criterion analogous to that obtained over $\mathbb{Q}$.
Even when this additional assumption is satisfied, the finite-field analogue may not always be as informative as the constraint over $\mathbb Q$ because an integer coefficient or a linear combination that is nonzero over $\mathbb Q$ may become zero after reduction modulo the characteristic, causing the corresponding constraint to become trivial. 

Moreover, the trace-based constraints provide sufficient conditions for invertibility, but they do not characterize invertibility completely. In particular, they do not address the necessary direction. To this end, we complement these results by considering the polynomial gcd criterion, which gives a necessary and sufficient condition for the invertibility of a circulant matrix over a finite field.

While this polynomial gcd criterion gives a complete characterization of invertibility, which is good in theory, in practice it may require explicit factorization or gcd computations in $F[x]$. By contrast, the trace-based criteria developed here often yield sufficient conditions that can be verified using only aggregate information about the first row. The trace theory and the polynomial gcd criterion used in this paper are standard tools, which we employ to establish our main results.

As a concrete illustration of our trace-based criterion, we show in Example~\ref{ex:inv} that for $q=7$ and $n=11^s$ with $s\ge 2$, the corresponding linear constraints of the \(0,1\)-circulant matrix depend only on aggregate quantities such as the total number of nonzero entries and the numbers of nonzero entries in prescribed divisibility classes. Hence the criterion applies to broad families of \(0,1\)-circulant matrices, rather than to isolated examples depending on a specific support pattern.

The paper is organized as follows. In Section 2, we develop trace-based criteria for invertibility over finite fields. We introduce finite-field analogues of Ramanujan sums and relate them to the trace conditions arising from Frobenius orbits. Under the condition $\mathrm{ord}_d(\mathrm{char}(F))=\phi(d)$, these trace conditions reduce to Ramanujan-sum congruences. Examples are included to illustrate both the limitations and the usefulness of the criteria: some rational-field constraints become trivial after reduction, while in other cases the finite-field conditions remain effective.

In Section 3, we discuss complete invertibility criteria via polynomial methods. We recall the polynomial gcd criterion and use it to derive rank and singularity consequences over arbitrary fields and finite fields. This section is not intended to replace the trace-based criteria of Section 2, but rather to provide the complete algebraic viewpoint and to record useful consequences for special cases, including \(0,1\)-circulant matrices, complements of such matrices, and the case where the size of the matrix is a power of the characteristic.

Throughout this paper, $\overline{F}$ denotes an algebraic closure of a field $F$, and for any positive integer $n \ge 2$ and an integer $a$ with $\gcd(a,n)=1$, 
we write $\mathrm{ord}_n(a)$ for the multiplicative order of $a$ modulo $n$, i.e.,
\[
\mathrm{ord}_n(a) := \min\{d \ge 1 : a^d \equiv 1 \pmod n\}.
\]
An integer $a$ with $\gcd(a,n)=1$ is called a primitive root modulo $n$ if $\mathrm{ord}_n(a)=\phi(n)$, where $\phi$ is Euler's totient function. Equivalently, the congruence class of $a$ modulo $n$ is a generator of the multiplicative group $(\mathbb Z/n\mathbb Z)^\times$, the group of units.

For $d\ge 1$ and $j \in \mathbb{Z}$, the Ramanujan sum is the integer
\[
R_d(j)\;:=\;\sum_{\substack{1\le a\le d\\\gcd(a,d)=1}} e^{2\pi i aj/d}\ \in\ \mathbb Z.
\]
For any field $F$, we write $\overline{R_d(j)}\in F$ for the image of $R_d(j)$ under the
natural ring homomorphism $\mathbb Z\to F$ by $n \longrightarrow n \cdot 1_F$, where $1_F$ is the unity of $F$. In particular, when $\mathrm{char}(F)>0$, we know that $\overline{R_d(j)}$ is the reduction of $R_d(j)$ modulo $\mathrm{char}(F)$.

\section{Invertibility Criteria via Trace and Ramanujan Sums}\label{sec:frob-trace}

Throughout this section, let $F=\mathbb{F}_q$  and let $n \ge 2$ be an integer such that $\gcd(\mathrm{char}(F),n)=1$. Let $C = C(a_0, \cdots, a_{n-1})$ over $F$ be an $n \times n$ circulant matrix with associated polynomial $f(x)=\sum_{j=0}^{n-1} a_j x^j$.


Fix a primitive $n$-th root of unity $\epsilon\in\overline{\mathbb{F}}_q$.
Note that $C$ is singular if and only if there exists $j\in\{0,1,\ldots,n-1\}$ such that
$f(\epsilon^j)=0$, since
\begin{equation}
\label{eq:det}
\det C \;=\; \prod_{j=0}^{n-1} f(\epsilon^j)\quad\text{in }\mathbb F_q
\end{equation}
(see, e.g., \cite[Section~2.1]{jitman2020}; see also
\cite[Section~2]{doublecirculant}). In this case, $\epsilon^j$ is a primitive $d$-th root of unity, where $d = n/\gcd(j,n)$. By letting $\zeta = \epsilon^j$ we obtain the following conclusion: $C$ is singular if and only if there exists $d \mid n$ and a primitive $d$-th root of unity $\zeta$ such that $f(\zeta) = 0$. We illustrate this criterion over a finite field with a concrete family.

\begin{example}
Let $p$ be an odd prime, and let $h\in \mathbb{F}_{p}$ be a primitive root modulo $p$. For each positive $m\ge 1$, consider the $(p-1) \times (p-1)$ circulant matrix $G_{p,m,h}=C(1,h^m, h^{2m}, \dots,h^{(p-2)m})$ over $\mathbb{F}_p$. Such matrices were introduced over $\mathbb{Q}$ in \cite{LMA} and their singularity was discussed there (see the paper for details). 

Since we immediately see that the matrices $G_{p,m,h}$ are all singular in $\mathbb{F}_{p}$ for any odd prime $p$ and a positive integer $m$, we consider the variant circulant matrices $K_{p,m,h}:=G_{p,m.h}-I_{p-1}$. For example, when $p=5$, we can pick $h=2$, and when $p=7$, $h=3$ can be used. Thus, 
\[
K_{5,m,2}=\left[\begin{array}{cccc} 
0&2^m&4^m&3^m\\ 
3^m&0&2^m&4^m\\
4^m&3^m&0&2^m\\
2^m&4^m&3^m&0
\end{array}\right] \hbox{in $\mathbb{F}_5$},\quad
K_{7,m,3}=\left[\begin{array}{cccccc} 
0&3^m&2^m&6^m&4^m&5^m\\ 
5^m&0&3^m&2^m&6^m&4^m\\
4^m&5^m&0&3^m&2^m&6^m\\
6^m&4^m&5^m&0&3^m&2^m\\
2^m&6^m&4^m&5^m&0&3^m\\
3^m&2^m&6^m&4^m&5^m&0
\end{array}\right] \hbox{in $\mathbb{F}_7$}.\]

Then, we can show that the matrix $K_{p,m,h}$ is invertible in $\mathbb{F}_{p}$. To see this, note that with $n:=p-1$, we have $\gcd(\mathrm{char}(\mathbb F_p), n) = \gcd(p, p-1)=1$ and the associated polynomial of $K_{p,m,h}$ is 
\[f(x)=\sum_{j=1}^{n-1}h^{mj} x^j.\] 
Note that all the $n$-th roots of unity are exactly the elements of the multiplicative group $\mathbb{F}_{p}^\times=\{1,2,\dots,p-1\}$. Now, pick any $\zeta\in\mathbb{F}_{p}^\times$. Since $h$ is a primitive element, $\zeta=h^k$ for some $0\le k\le n-1$. Thus, 
\[f(\zeta)=f(h^k)=\sum_{j=1}^{n-1}h^{mj} h^{kj}=\sum_{j=1}^{n-1}(h^{m+k})^j.\]
Let $r:=h^{m+k}$.
If $r=1$, then $f(\zeta)=\sum_{j=1}^{n-1}r^j=p-2\ne 0$ in $\mathbb F_p$. If $r\ne 1$, then 
\[f(\zeta)=\sum_{j=1}^{n-1}r^j=\frac{r^n-1}{r-1}-1
=\frac{(h^{m+k})^{p-1}-1}{h^{m+k}-1}-1=0-1\ne 0,\quad \mbox{in $\mathbb F_p$}.\]
Thus, the determinant of the matrix $K_{p,m,h}$ is nonzero by (\ref{eq:det}), and hence is invertible in $\mathbb{F}_{p}$ for any odd prime $p$ and positive integer $m$.
\end{example}

The example above is simple because all the relevant roots of unity already lie in $\mathbb{F}_q$. In general, they lie in extensions of $\mathbb{F}_q$, and we treat this case systematically in the next subsection.

\subsection{Trace-Based Criteria for Invertibility}

Let $\sigma$ be the Frobenius automorphism defined by
\[
\sigma:\overline{\mathbb{F}}_q\to\overline{\mathbb{F}}_q,\qquad x\mapsto x^{\,q}.
\]
It is well-known that $\sigma$ fixes $\mathbb F_q$. For every $\alpha\in\overline{\mathbb{F}}_q$, the (Frobenius) orbit of $\alpha$ is
\[
O(\alpha):=\{\sigma^i(\alpha) : i \ge 0\}
=\{\alpha^{q^i} : i \ge 0\}.
\]
Since $\alpha$ is algebraic over $\mathbb F_q$, $O(\alpha)$ is finite; its size is the least $t \ge 1$ such that
$\sigma^t(\alpha)=\alpha$, and equals $[\mathbb F_q(\alpha):\mathbb F_q]$.
In particular, if $\alpha=\zeta$ is a primitive $d$-th root of unity, then
\[
\sigma^t(\zeta)=\zeta
\iff \zeta^{q^t}=\zeta
\iff \zeta^{q^t-1}=1
\iff d \mid (q^t-1)
\iff q^t \equiv 1 \pmod d.
\]
Since we assume $\gcd(\mathrm{char}(\mathbb F_q), n) = 1$, we have $\gcd(q, d) = 1$ for any divisor $d$ of $n$. Thus, the smallest such $t \ge 1$ is given by $t = |O(\zeta)| = \mathrm{ord}_d(q)$ (\cite{ln}).

\begin{theorem}\label{thm:fb}
Let $C=C(a_0,a_1,\ldots,a_{n-1})$ over $\mathbb F_q$ with $\gcd(\mathrm{char}(\mathbb F_q), n) = 1$. 
If $C$ is singular over $\mathbb F_q$, then there exist a divisor $d\mid n$ and the orbit $O$
of a primitive $d$-th root of unity in $\overline{\mathbb{F}}_q$ such that
\begin{equation}
    \label{eq:thm1}
\sum_{j=0}^{n-1}\Bigl(\sum_{\gamma\in O}\gamma^{\,j}\Bigr)a_j=0
\quad\text{in }\mathbb F_q.
\end{equation}
\end{theorem}


\begin{pf}
Fix a primitive $n$-th root of unity $\epsilon\in\overline{\mathbb{F}}_q$. 
Since $C$ is singular over $\mathbb F_q$, the associated polynomial $f(x) \in \mathbb{F}_q[x]$ satisfies $f(\epsilon^j)=0$ for some
$j\in\{0,1,\ldots,n-1\}$. Let $\zeta =\epsilon^j$. Then $\zeta$ is a primitive $d$-th root of unity,
where $d=n/\gcd(j,n)$ is a divisor of $n$.

Let $\sigma$ be the Frobenius automorphism.
We observe that $\sigma(f(\alpha))=f(\sigma(\alpha))$ for all $\alpha\in\overline{\mathbb{F}}_q$, since $\sigma$ fixes $\mathbb F_q$. 
Consequently, since $f(\zeta) = 0$, we have $f(\sigma^i(\zeta)) = \sigma^i(f(\zeta)) = 0$ for all $i \ge 0$. This means that for $t=\mathrm{ord}_d(q)$, every element in the orbit \( O:=O(\zeta) = \{ \zeta, \sigma(\zeta), \dots, \sigma^{t-1}(\zeta) \} \) is a root of $f(x)$.

Summing the quantities $f(\sigma^i(\zeta))=0$ over $i=0,1,\ldots,t-1$ and expanding, we obtain
\[
0= \sum_{i=0}^{t-1}f(\sigma^i(\zeta))
=\sum_{j=0}^{n-1} a_j\Bigl(\sum_{i=0}^{t-1}(\zeta^{q^i})^j\Bigr) 
 = \sum_{j=0}^{n-1} a_j\Bigl(\sum_{\gamma \in O}\gamma^{j}\Bigr).
\]
Here, the last equality occurs because the orbit $O$ consists of exactly the elements $\zeta^{q^i}$ for $i=0, \dots, t-1$.
\hfill $\square$
\end{pf}

The orbit sum in (\ref{eq:thm1}) is in fact a field trace, as we note next.

\begin{remark}\label{rem:1}
Let $\mathbb F_{q^m}$ be an extension of $\mathbb F_q$. For $\alpha \in \mathbb F_{q^m}$, the \emph{trace} of $\alpha$ over $\mathbb F_q$ is defined by
\[
\mathrm{Tr}_{\mathbb F_{q^m}/\mathbb F_q}(\alpha)
=\sum_{i=0}^{m-1}\sigma^i(\alpha)
=\sum_{i=0}^{m-1}\alpha^{q^i}
=\alpha+\alpha^q+\cdots+\alpha^{q^{m-1}},
\]
where $\sigma(x)=x^q$. Taking $m=t$ and $\alpha=\zeta^j$ in the notation of Theorem~\ref{thm:fb} and its proof, we have, for every $j=0,1,\ldots,n-1$,
\begin{equation}
\label{eq:trace}
\mathrm{Tr}_{\mathbb F_{q^t}/\mathbb F_q}(\zeta^j)=\sum_{i=0}^{t-1}(\zeta^j)^{q^i}=\sum_{\gamma\in O}\gamma^{\,j},
\end{equation}
since $\zeta\in\mathbb F_{q^t}$ (and hence $\zeta^j\in\mathbb F_{q^t}$ for all $j$), and $O=\{\zeta,\zeta^q,\ldots,\zeta^{q^{t-1}}\}$. Thus, condition (\ref{eq:thm1}) of Theorem~\ref{thm:fb} can be written as
\[\sum_{j=0}^{n-1}\Bigl(\mathrm{Tr}_{\mathbb F_{q^t}/\mathbb F_q}(\zeta^j)\Bigr)a_j=0
\quad\text{in }\mathbb F_q. \]

Finally, this trace value does not depend on the chosen representative of $O$: if $\zeta'=\zeta^{q^k}\in O$ with $0\le k\le t-1$, then
\[
\mathrm{Tr}_{\mathbb F_{q^t}/\mathbb F_q}(\zeta'^j)=\sum_{i=0}^{t-1}(\zeta'^{\,j})^{q^i}
=\sum_{i=0}^{t-1}\bigl((\zeta^{\,j})^{q^k}\bigr)^{q^i}
=\sum_{i=0}^{t-1}(\zeta^{\,j})^{q^{i+k}}
=\sum_{l=0}^{t-1}(\zeta^{\,j})^{q^l}=\mathrm{Tr}_{\mathbb F_{q^t}/\mathbb F_q}(\zeta^j),
\]
where the indices $l = i+k$ are taken modulo $t$.
\end{remark}

Hence, we obtain the following corollary.

\begin{corollary}\label{cor:suff}
Under the assumptions of Theorem~\ref{thm:fb}, if for every divisor $d\mid n$ and every orbit $O$
of primitive $d$-th roots of unity, there exists an element $\zeta\in O$ such that
\begin{equation}
\label{eq:cor}
\sum_{j=0}^{n-1}\Bigl(\mathrm{Tr}_{\mathbb F_q(\zeta)/\mathbb F_q}(\zeta^j)\Bigr)a_j \neq 0
\quad\text{in }\mathbb F_q,
\end{equation}
then $C$ is invertible over $\mathbb F_q$.
\end{corollary}

Corollary~\ref{cor:suff} applies in general, but the trace values $\mathrm{Tr}_{\mathbb F_q(\zeta)/\mathbb F_q}(\zeta^j)$ appearing in it are left implicit. We now show that when ${\rm ord}_d(q)=\phi(d)$, each such trace equals the reduction $\overline{R_d(j)}$ of the classical Ramanujan sum $R_d(j)$; since these values are explicitly known, the invertibility condition becomes an explicit linear condition on the entries $a_0, a_1, \dots, a_{n-1}$ of $C$.

\subsection{Ramanujan Sum Criteria for Special Cases}

Assume $q$ is a prime power with $\gcd(q,d)=1$ and let $\zeta$ be a primitive $d$-th root of unity in some extension field of
$\mathbb F_q$. Motivated by the classical Ramanujan sum $R_d(j)$, namely the sum of the $j$-th powers of the primitive $d$-th roots of unity, we consider
\[
S_d(j): = \sum_{\substack{1\le a\le d\\\gcd(a,d)=1}}\zeta^{aj}, \quad T_d(j) : = \sum_{a = 1}^d \zeta^{aj}. 
\]
It is easy to see that $T_d(j) = \begin{cases}d & \text{ if } d \mid j \\ 0 & \text{ otherwise.} \end{cases}$

\begin{theorem}\label{thm:sdj}
The sum $S_d(j)$ is equal to the reduction of the classical Ramanujan sum $R_d(j)$ modulo $\mathrm{char}(\mathbb F_q)$, i.e., for every $d \ge 1$ and integer $j$, $S_d(j) = \overline{R_d(j)}$.
\end{theorem}
\begin{pf}
Observe that 
\[ \sum_{\substack{1 \le a \le d \\ \gcd(a, d) = m}} \zeta^{aj} = \sum_{\substack{1 \le b \le d/m \\ \gcd(b, d/m) = 1}} (\zeta^m)^{bj}\]
and \[ T_d(j) = \sum_{m \mid d}\sum_{\substack{1 \le a \le d \\ \gcd(a, d) = m}} \zeta^{aj} = \sum_{m \mid d} \sum_{\substack{1 \le b \le d/m \\ \gcd(b, d/m) = 1}} (\zeta^m)^{bj} = \sum_{r \mid d} \sum_{\substack{1 \le b \le r \\ \gcd(b, r) = 1}}(\zeta^{\frac{d}{r}})^{bj} = \sum_{r \mid d}S_r(j).\]

Applying M\"obius inversion, we obtain 
\[ S_d(j) = \sum_{t \mid d}  \mu\!\left(\frac{d}{t}\right) T_t(j)=  \sum_{t \mid \gcd(d,j)} \mu\!\left(\frac{d}{t}\right)\,t =\overline{R_d(j)}.\]

The last equality is obtained by reducing (modulo $\mathrm {char}(\mathbb F_q)$) the well-known identity, which expresses Ramanujan sum in terms of the M\"obius function (see \cite[Theorem~8.6]{apostol}).
\hfill $\square$
\end{pf}

\begin{corollary}\label{cor:ra}
Let $O$ be the Frobenius orbit of a primitive $d$-th root $\zeta$, namely
\[
O=\{\zeta,\zeta^{q},\zeta^{q^2},\ldots,\zeta^{q^{t-1}}\},
\qquad t={\mathrm{ord}}_d(q).
\] For positive integers $d$ and $q$ with $\gcd(d, q) = 1$, if ${\mathrm{ord}}_d(q) = \phi(d)$, then for every integer $j$ \[\mathrm{Tr}_{\mathbb F_q(\zeta)/\mathbb F_q}(\zeta^j)=\sum_{\gamma\in O}\gamma^{\,j} = \overline{R_d(j)}.\]
\end{corollary}

\begin{pf}
For each integer $j$, the first equality is immediate from (\ref{eq:trace}). 
For the second equality, by using (\ref{eq:trace}) again,
\[
\sum_{\gamma\in O}\gamma^{\,j}
=\sum_{i=0}^{t-1}(\zeta^{\,j})^{q^i}.
\]
Since ${\mathrm{ord}}_d(q)=\phi(d)$, $q$ generates $(\mathbb Z/d\mathbb Z)^\times$ and the set
of primitive $d$-th roots of unity forms a single orbit. Equivalently,
\[
O=\{\zeta^{\,a} : 1\le a\le d,\ \gcd (a,d)=1\}.
\]
Therefore,
\[
\sum_{\gamma\in O}\gamma^{\,j}
=\sum_{\substack{1\le a\le d\\\gcd(a,d)=1}} (\zeta^{\,a})^{\,j}
= S_d(j)=\overline{R_d(j)}.
\]
\hfill $\square$
\end{pf}

By Corollary~\ref{cor:suff} and Corollary~\ref{cor:ra}, we obtain the following theorem.

\begin{theorem}\label{thm:ra}
Let $C=C(a_0,a_1,\ldots,a_{n-1})$ over $\mathbb F_q$ with $\gcd(\mathrm{char}(\mathbb F_q), n) = 1$. 
Assume that $\sum_{j = 0}^{n-1} a_j \ne 0$.
If, for every divisor $d\mid n$ with $d\ne 1$, it satisfies $\mathrm{ord}_d(q)=\phi(d)$ and 
\begin{equation}
\label{eq:thm3}
\sum_{j=0}^{n-1} \overline{R_d(j)}\,a_j \neq 0 \quad\text{in }\mathbb F_q,
\end{equation}
then $C$ is invertible.   
\end{theorem}

Note that the linear constraints in (\ref{eq:thm3}) coincide with those in \cite{LMA} after
reduction modulo $\mathrm{char}(\mathbb F_q)$.
\begin{theorem}\label{thm:q}\cite[Corollary 3.4]{LMA}
Let $C=C(a_0,a_1,\ldots,a_{n-1})$ over $\mathbb{Q}$. Assume $\sum_{j = 0}^{n-1} a_j \ne 0$. If, for every divisor $d \mid n$ with $d\ne 1$,
\[
\sum_{j=0}^{n-1} R_d(j)\,a_j \neq 0,
\]
then $C$ is invertible.   
\end{theorem}

\medskip
By \cite[Proposition~4.1.3]{ire}, primitive roots modulo $d$ exist if and only if
$d\in\{2,4,p^k,2p^k\}$ for some odd prime $p$ and $k\ge 1$.

\begin{theorem}\label{thm:primepower}
Let $p$ be an odd prime, let $q$ be a prime power with $\gcd(p, q) = 1$, and let $k \ge 1$.
\begin{enumerate}
\item[\textnormal{(i)}] Assume $\mathrm{ord}_{p^k}(q)=\phi(p^k)=p^{k-1}(p-1)$.
Let $L=\mathbb F_{q^{\phi(p^k)}}$ and $\zeta\in L$ be a primitive $p^k$-th root of unity. Then for every integer $j$,
\[
\mathrm{Tr}_{L/\mathbb F_q}(\zeta^{j})
=
\begin{cases}
p^{k-1}(p-1), & p^k\mid j,\\[2mm]
-\,p^{k-1}, & p^{k-1}\mid j\ \text{but}\ p^k\nmid j,\\[2mm]
0, & p^{k-1}\nmid j,
\end{cases}
\qquad\text{in }\mathbb F_q.
\]
\item[\textnormal{(ii)}] Assume $\mathrm{ord}_{2p^k}(q)=\phi(2p^k)=p^{k-1}(p-1)$.
Let $L=\mathbb F_{q^{\phi(2p^k)}}$ and $\zeta\in L$ be a primitive $2p^k$-th root of unity. Then for every integer $j$, 
\[
\mathrm{Tr}_{L/\mathbb F_q}(\zeta^{j})
=
\begin{cases}
p^{k-1}(p-1), & 2p^k\mid j,\\[2mm]
-\,p^{k-1}, & 2p^{k-1}\mid j \ \text{but} \ 2p^k\nmid j,\\[2mm]
p^{k-1}, & p^{k-1}\mid j,\ p^k\nmid j,\ \text{and } 2\nmid j, \\[2mm]
 -p^{k-1}(p-1), & p^k \mid j \ \text{ and } \ 2 \nmid j, \\[2mm]
0, & p^{k-1}\nmid j,
\end{cases}
\qquad\text{in }\mathbb F_q.
\]
\end{enumerate}
\end{theorem}

\begin{pf}
Note that if $\textrm{ord}_d(q) = \phi(d)$ with $d = p^k$ or $2p^k$ for some odd prime~$p$ and $k\ge 1$, then by Corollary~\ref{cor:ra}, we obtain explicit traces using the known values and properties of the Ramanujan sum (\cite{nicol62, lucht10, anderson53, hardy27}).

For (i), with an odd prime $p$ and $k \ge 1$, we have
\[
R_{p^k}(j)
=
\begin{cases}
\phi(p^k) = p^{k-1}(p-1), & p^k\mid j,\\[2mm]
-\,p^{k-1}, & p^{k-1}\mid j \ \text{but}\ p^k\nmid j,\\[2mm]
0, & p^{k-1}\nmid j ,
\end{cases}
\]
and (i) is immediate from Corollary~\ref{cor:ra}, which states that 
\[
\mathrm{Tr}_{L/\mathbb F_q}(\zeta^j)=\overline{R_{p^k}(j)}.
\]

For (ii), 
we use
\[
R_{2p^k}(j)=R_2(j)\,R_{p^k}(j),
\]
which is obtained from the fact that if $\gcd(r,s) = 1$, then 
\( R_{rs}(j) = R_r(j)R_s(j)\) for every $j$.
Thus
\[
\mathrm{Tr}_{L/\mathbb F_q}(\zeta^j)=\overline{R_{2p^k}(j)}
= \overline{R_2(j)\,R_{p^k}(j)},
\]
and the stated formula follows by combining the explicit expressions
for $R_2(j)$ and $R_{p^k}(j)$.
\hfill $\square$
\end{pf}

\medskip
Unlike the case $\mathbb{Q}$, over a finite field, to replace $\sum_{\gamma\in O}\gamma^{\,j}$ in (1) by $\overline{R_d(j)}$ in (4) for Theorem~\ref{thm:ra}, the condition $\mathrm{ord}_d(q)=\phi(d)$ is necessary. 
In particular, the following example shows that, when $\mathrm{ord}_d(q) \neq \phi(d)$, one may have  $\sum_{\gamma\in O}\gamma^j \neq \overline{R_d(j)}$ for some $j$. 

\begin{example}
If $d = 5$ and $q = 4$, the $\phi(5) = 4$ and $\mathrm{ord}_5(4)=2.$
Let $\zeta$ be a primitive $5$-th root of unity, so its orbit is
\[
O=\{\zeta,\zeta^4\}.
\]
For $j=1$,
\[
\sum_{\gamma\in O}\gamma
=
\zeta+\zeta^{4}. \]
Let $\eta = \zeta+\zeta^4$. Then
\[ \eta^2+\eta+1 = (\zeta^4+\zeta)^2 + (\zeta^4+\zeta) + 1 = \zeta^3+\zeta^2+\zeta^4+\zeta+1 = 0\   \text{in }\mathbb F_4.\]
But if $\eta = \overline{R_5(1)} = \overline{-1} = 1$, then $\eta^2+\eta+1 = 1$ in $\mathbb F_4$. 
\end{example}
\medskip

Another contrast with $\mathbb Q$ is that, even when the trace admits a Ramanujan-sum expression, its behavior over $\mathbb F_q$ may differ. Indeed, reduction modulo \(\mathrm{char}(\mathbb{F}_q)\) may annihilate the coefficients coming from the Ramanujan sum, so a nontrivial linear constraint over \(\mathbb{Q}\) may become trivial over \(\mathbb{F}_q\). If only part of them vanish, the condition is weakened, but may still provide useful information while being easier to satisfy.

\medskip
\begin{example}
Take $q=3$, $d=7$, and $n=7$.
Then $\mathrm{ord}_7(3)=\phi(7)=6$,
so we are in the single-orbit case.
Since
\[
R_7(j)\in\{6,-1\},
\]
reduction modulo $3$ yields
\[
\overline{R_7(j)}=
\begin{cases}
0 & \text{ if  } 7\mid j,\\
2 & \text{ if  } 7\nmid j.
\end{cases}
\]

For the identity matrix \(I_7\) over any field \(F\), invertibility is obvious. Also,
\[
\sum_{i=0}^6 a_i = 1,
\]
which is nonzero in \(F\).
If \(F=\mathbb{Q}\), then the linear constraint for \(d=7\) becomes
\[
\sum_{\substack{0\le j\le 6\\7\mid j}}6a_j-\sum_{\substack{0\le j\le 6\\7\nmid j}} a_j
=6a_0=6\neq 0.
\]
Thus, Theorem~\ref{thm:q} correctly detects the invertibility of \(I_7\).
On the other hand, if \(F=\mathbb{F}_3\), then
\[
\sum_{\substack{0\le j\le 6\\7\mid j}}6a_j-\sum_{\substack{0\le j\le 6\\7\nmid j}} a_j
=-\sum_{\substack{0\le j\le 6\\7\nmid j}} a_j=0
\quad\text{in }\mathbb F_3.
\]
Hence, this linear constraint test does not detect the invertibility of \(I_7\) over \(\mathbb{F}_3\).
\end{example}

 Nevertheless, Theorem~\ref{thm:ra} remains a powerful tool, since the corresponding linear constraints are obtained directly from the known values of the Ramanujan sums. In particular, invertibility can often be determined solely from simple statistics of the entries of the first row, without imposing any specific arrangement or pattern on those entries. The following example illustrates this point.

\begin{example}
\label{ex:inv}
Consider \(q=7\) and \(n=11^s\), where \(s\ge 2\). Let
\[
C=C(a_0,a_1,\dots,a_{11^s-1})
\]
be a circulant matrix over \(\mathbb F_7\) with \(a_j\in\{0,1\}\).

Since \(7\) is a primitive root modulo \(11^2\), it is a primitive root
modulo \(11^k\) for every \(k\ge 1\). This follows from a standard result
for primitive roots modulo odd prime powers \cite{cohen}. In other words,
$\mathrm{ord}_d(7)=\phi(d)$
 for every \(d\mid 11^s\) with \(d\ne 1\). 

By Theorem~\ref{thm:ra}, \(C\) is invertible provided that
\[
\sum_{j=0}^{11^s-1} a_j \neq 0
\quad\text{in }\mathbb F_7,
\]
and, for \(d=11^k\) with \(1\le k\le s\), 
\[
\sum_{j=0}^{11^s-1}\overline{R_{11^k}(j)}\,a_j
\ne 0
\quad\text{in }\mathbb F_7.
\]
By Theorem~\ref{thm:primepower}(i) (with $p=11$ and $q=7$), 
\[
\overline{R_{11^k}(j)}
=
\begin{cases}
11^{k-1}\cdot 10, & 11^k\mid j,\\[2mm]
-\,11^{k-1}, & 11^{k-1}\mid j\ \text{but}\ 11^k\nmid j,\\[2mm]
0, & 11^{k-1}\nmid j,
\end{cases}
\qquad\text{in }\mathbb F_7.
\]
Writing 
\[
M_k:=\sum_{\substack{0\le t\le 11^s-1\\ 11^k\mid t}} a_t
\qquad (0\le k\le s),
\]
we obtain 
\[
\sum_{t=0}^{11^s-1}\overline{R_{11^k}(t)}\,a_t
=
4^{\,k-1}\,(4M_k-M_{k-1})
\quad\text{in }\mathbb F_7,
\]
since $11\equiv 4 \pmod 7$.

Since $4^{\,k-1}\in \mathbb F_7^\times$, the constraint
\[
\sum_{j=0}^{11^s-1}\overline{R_{11^k}(j)}\,a_j \neq 0
\]
is equivalent to
$4M_k-M_{k-1}\neq 0
\quad\text{in }\mathbb F_7.$

Now fix \(a_0=1\). Define
\[
S:=|\{j: a_j=1\}|,
\]
and, for \(1\le k\le s-1\),
\[
T_k:=|\{\,m\in\{1,\dots,11^{s-k}-1\}: a_{11^k m}=1\,\}|.
\]
Then
\[
M_0=S,\qquad
M_k=1+T_k\quad (1\le k\le s-1),
\qquad
M_s=1.
\]

Suppose that
\[
S\not\equiv 0 \pmod 7,
\qquad
S\equiv 1+T_1\equiv \cdots \equiv 1+T_{s-1}\pmod 7,
\qquad
T_{s-1}\not\equiv 3\pmod 7.
\]
Then
\[
4M_k-M_{k-1}\neq 0
\quad\text{in }\mathbb F_7,
\quad \forall\,1\le k\le s.
\]
Indeed, for \(1\le k\le s-1\), if
\[
S\equiv 1+T_1\equiv \cdots \equiv 1+T_{s-1}\pmod 7,
\]
then 
\[
4M_k-M_{k-1}= 4S-S=3S\neq 0
\quad\text{in }\mathbb F_7,
\]
and for \(k=s\),
\[
4M_s-M_{s-1}
=
4-(1+T_{s-1})
=
3-T_{s-1}\neq 0
\quad\text{in }\mathbb F_7.
\]

Therefore the corresponding \(0,1\)-circulant matrix is invertible over
\(\mathbb F_7\).

Consequently, these sufficient congruence conditions depend only on the
aggregate quantities \(S,T_1,\dots,T_{s-1}\), and not on the specific
locations of the nonzero entries. Hence any \(0,1\)-circulant matrix with
\(a_0=1\) whose associated values \(S,T_1,\dots,T_{s-1}\) satisfy these
conditions is invertible over \(\mathbb F_7\).

This shows that, for \(n=11^s\) with \(s\ge 2\), the sufficient criterion
applies to a broad and flexible family of \(0,1\)-circulant matrices.
The number of such admissible patterns increases rapidly with \(s\).
\end{example}

This example highlights a clear advantage of the sufficient condition: it yields a pattern-independent invertibility criterion, rather than on its precise arrangement, and is therefore very easy to apply in practice. Although the preceding results provide such sufficient conditions for invertibility, the next section establishes a necessary and sufficient criterion.

\section{Complete Invertibility Criteria via Polynomial Methods}

\subsection{Singularity Criteria via Simple Coefficient Conditions}
\label{sec:singularity}

We begin with two general results over an arbitrary field, which will serve as the algebraic foundation for the criteria developed below, particularly in the finite field setting.

The following result is well known; hence we skip its proof. 

\begin{proposition}\label{prop:min}
Let $F$ be a field, $\overline{F}$ an algebraic closure of $F$, and let $\alpha \in \overline{F}$.
Let $g(x)$ be the minimal polynomial of $\alpha$ over $F$.
For any polynomial $f(x) \in F[x]$, if $f(\alpha) = 0$, then $g(x) \mid f(x)$ in $F[x]$.
\end{proposition}



The following lemma is known in various forms. We state it here in a form convenient for our purposes and include a short proof. In particular, it yields a purely algebraic proof of the rank identity over an arbitrary field, recovering the classical result of Ingleton~\cite{ingleton} in the case $F=\mathbb{Q}$.

\begin{lemma}\label{lem:rank_circulant_any_field}
Let $C = C(a_0,a_1,\dots,a_{n-1})$ over a field $F$, and
let
\[
f(x)=a_0+a_1x+\cdots+a_{n-1}x^{n-1}\in F[x]
\]
be its associated polynomial.
Then
\[
\rank_F C=n-\deg\gcd\bigl(f(x),x^n-1\bigr).
\]
In particular, $C$ is singular over $F$ if and only if $\gcd(f(x),x^n-1)\neq 1$ in $F[x]$. Throughout, we take the $\gcd$ to be monic.
\end{lemma}

\begin{pf}
It is convenient to work with the circulant matrix $\widetilde C$ whose first row is
$(a_0,a_{n-1},a_{n-2},\ldots,a_1)$.
It is immediate that $C$ and $\widetilde C$ are permutation-similar (hence have the same rank and
the same singularity behavior).

Consider the quotient ring $R=F[x]/\langle x^n-1 \rangle$, and the $F$-linear map
\[
T_f : R \longrightarrow R,\qquad g(x)\longmapsto f(x)g(x)\bmod (x^n-1).
\]
With respect to the basis $\{1,x,x^2,\ldots,x^{n-1}\}$ of $R$, the matrix of $T_f$ is precisely
$\widetilde C$.

Let $u(x)=\gcd\bigl(f(x),x^n-1\bigr)$ in $F[x]$, and write $f(x)= u(x) f_1(x)$ and $x^n-1=u(x) h_1(x)$ with
$\gcd(f_1(x),h_1(x))=1$.
Then for $g\in R$ we have
\[
T_f(g)=0 \iff (x^n-1)\mid f(x) g(x) \iff u(x) h_1(x) \mid u(x) f_1(x) g(x)  \iff h_1(x) \mid g(x). 
\]
Hence, $\ker(T_f)=\langle h_1 \rangle$, the ideal generated by $h_1$ in $R$.

Each element in $\langle h_1 \rangle$ is $h_1g$ for some $g(x) \in F[x]$, where $g(x) = u(x)g_1(x) + v(x)$ with either $v = 0$ or $ \deg v(x)<\deg u(x)$. Therefore, $h_1(x)g(x) = u(x)h_1(x)g_1(x)+h_1(x)v(x) = (x^n-1)g_1(x)+h_1(x)v(x) = h_1(x)v(x) \in R$. Also, for any nonzero polynomial $v(x)$ with $ \deg v(x) < \deg u(x)$, if $h_1(x)v(x) = 0$, then $(x^n-1) \mid h_1(x)v(x)$ and so $u(x) \mid v(x)$ which is a contradiction. Hence, every nonzero element of $\langle h_1(x)\rangle$ has a unique representative of the form $h_1(x)v(x)$ with $\deg v(x)<\deg u(x)$, so $\nullity(T_f) = \dim_F\ker(T_f)=\deg u(x)$.
Therefore,
\[
\rank_F \widetilde C=\rank_F(T_f)=n-\deg u(x)
= n-\deg\gcd\bigl(f(x),x^n-1\bigr).
\]
Since $\rank_F C=\rank_F \widetilde C$, the claim follows.
\hfill $\square$
\end{pf}

Let $J$ be the $n \times n$ matrix with all entries $1$.

\begin{corollary}\label{cor:rank}
     Let $C = C(a_0,a_1,\dots,a_{n-1})$ be a circulant $0,1$-matrix over a field $F$. Let $f(x)$ be the associated polynomial of $C$ and let $g(x)$ be the associated polynomial of $ J - C$. If $x - 1$ divides neither $f(x)$ nor $g(x)$, then  
\[ \rank_F(C) = \rank_F(J-C). \]

\end{corollary}

\begin{pf}
Note that $\gcd(f(x), x^n-1) = \gcd(f(x), x^{n-1}+x^{n-2}+\cdots + x +1)$ and $\gcd(g(x), x^n-1) = \gcd(g(x), x^{n-1}+x^{n-2}+\cdots + x +1)$ since $(x - 1) \nmid f(x)$ and $(x-1) \nmid g(x)$. Hence, we obtain
\[ \gcd(f(x), x^n - 1) = \gcd(f(x), f(x)+g(x)) = \gcd(g(x), g(x)+f(x)) = \gcd(g(x), x^n-1). \]
Hence, $\rank_F(C) = \rank_F(J-C)$ by Lemma~\ref{lem:rank_circulant_any_field}.
\hfill $\square$
\end{pf}

\begin{corollary}
Let $C = C(a_0,a_1,\dots,a_{n-1})$ be a circulant $0,1$-matrix over a field $F$. Let $f(x)$ be the associated polynomial of $C$ and let $g(x)$ be the associated polynomial of $ J - C$.  Assume that either $\mathrm{char}(F) = 0$ or $\mathrm{char}(F) = p > 0$ and $p \mid n$. Then $C$ is invertible if and only if $J-C$ is invertible.  
\end{corollary}

\begin{pf}
Suppose that $\mathrm{char}(F) = 0$. If $f(1) = 0$ or $f(1) = n$, then $C$ is singular since it is either $C$ is the zero matrix or $C = J$. Hence, $J - C$ is also singular in this case. Otherwise, $1\le f(1) \le n-1$, and thus $1\le g(1) \le n-1$ since $g(1) = n-f(1)$. Hence, $(x - 1) \nmid f(x)g(x)$ in this case, and by Corollary~\ref{cor:rank}, $C$ is invertible if and only if $J-C$ is invertible.

Suppose that $\mathrm{char}(F) = p > 0$ and $p \mid n$. If $p \mid f(1)$, then $(x - 1) \mid f(x)$ in $F[x]$. Hence $\gcd(f(x), x^n-1) \neq 1$ and $C$ is singular. In this case, $(x -1) \mid g(x)$ since $p \mid g(1)$. Hence $J-C$ is singular.  If $p \nmid f(1)$, then $(x - 1) \nmid f(x)$ and $p \nmid g(1)$. Hence, $(x - 1)\nmid g(x)$. By Corollary~\ref{cor:rank}, $C$ is invertible if and only if $J-C$ is invertible. 
\hfill $\square$
\end{pf}

\begin{corollary}
Let $C = C(a_0,a_1,\dots,a_{n-1})$ be a circulant $0,1$-matrix over a field $F$. Let
\[
m=\left|\{\, i : a_i=1,\ 0\le i\le n-1 \,\}\right|.
\]
Assume that 
\(\mathrm{char}(F)=p>0\) and \(p\nmid m(n-m)\). Then $C$ is invertible if and only if $J-C$ is invertible.
\end{corollary}

\begin{pf}
Let $f(x)$ and $g(x)$ be the associated polynomials of $C$ and $J-C$, respectively. Then $m = f(1) \neq 0$ and $n - m = g(1) \neq 0$, which imply that $(x - 1) \nmid f(x)g(x)$. By Corollary~\ref{cor:rank}, $C$ is invertible if and only if $J-C$ is invertible. 
\hfill $\square$
\end{pf}

\begin{proposition}\label{pro:ex}
Assume $q=p^k>2$ and $n\ge t(q-1)$ for some integer $t\ge 0$.
Suppose that among $a_0,\dots,a_{n-1}\in\mathbb{F}_q$, each element of $\mathbb{F}_q^\times$ appears exactly $t$ times
(in an arbitrary order), and the remaining $n-t(q-1)$ coefficients are $0$.
Then $C(a_0, a_1, \cdots, a_{n-1})$ is singular over $\mathbb{F}_q$.
\end{proposition}

\begin{pf}
By Lemma~\ref{lem:rank_circulant_any_field}, it suffices to show that $\sum_{j=0}^{n-1} a_j=0$ in $\mathbb{F}_q$.
By the assumption,
\[
\sum_{j=0}^{n-1} a_j \;=\; t\sum_{a\in\mathbb{F}_q^\times} a.
\]
Let $S=\sum_{a\in\mathbb{F}_q^\times} a$. For any $b\in\mathbb{F}_q^\times$, multiplication by $b$ permutes
$\mathbb{F}_q^\times$, hence
\[
S=\sum_{a\in\mathbb{F}_q^\times} a=\sum_{a\in\mathbb{F}_q^\times} ba=bS.
\]
Choose $b\in\mathbb{F}_q^\times\setminus\{1\}$, which exists since $q>2$. Then $(b-1)S=0$, and since $b-1\neq 0$
we conclude that $S=0$. Therefore $\sum_{j=0}^{n-1} a_j=tS=0$, as claimed.
\hfill $\square$
\end{pf}

\begin{theorem}\label{thm:n=ps}
Let $p$ be a prime, $q=p^k$, and $\mathbb{F}_q$ be a finite field.  Let $n = p^s, \ s\ge 1$. 
Let $C = C(a_0,a_1,\dots,a_{n-1})$ over a field $\mathbb{F}_q$.
Then the following are equivalent:
\begin{enumerate}
\item[\textnormal{(i)}] $C$ is singular over $\mathbb{F}_q$;
\item[\textnormal{(ii)}] $\sum_{j = 0}^{n-1}a_j = 0$ in $\mathbb{F}_q$.
\end{enumerate}
\end{theorem}

\begin{pf}
Let
\(
f(x)=a_0+a_1x+\cdots+a_{n-1}x^{n-1}
\)
be the associated polynomial of $C$.
Then, by Lemma \ref{lem:rank_circulant_any_field}, $C$ is singular over $\mathbb{F}_q$ if and only if 
\[
\gcd\bigl(f(x),x^n-1\bigr)\ne 1
\quad\text{in } \mathbb{F}_q[x].
\]
Since $\mathrm{char}(\mathbb{F}_q)=p$ and $n=p^s$, we have
\[
x^n-1=x^{p^s}-1=(x-1)^{p^s}\quad\text{in }\mathbb{F}_q[x].
\]
Hence, the only distinct irreducible factor of $x^n-1$ in $\mathbb{F}_q[x]$ is $(x-1)$, and therefore
\[
\gcd\bigl(f(x),x^n-1\bigr)\ne 1
\quad\Longleftrightarrow\quad
(x-1)\mid f(x) \Longleftrightarrow\quad \sum_{j = 0}^{n-1}a_j = 0 \] in $\mathbb{F}_q$.
\hfill $\square$
\end{pf}

\begin{example}
Let $p$ be a prime. Consider an $n \times n$ circulant $0,1$-matrix $C$ with $n = p^s$ and $m$ ones in its first row. Such matrices were studied over $\mathbb{Q}$ in \cite{LMA} and shown to be nonsingular if either $1 \le m \le p-1$ or $1 \le p^s-m \le p-1$.

We apply Theorem \ref{thm:n=ps} to determine the singularity of the above matrices over finite fields, rather than over $\mathbb{Q}$. Specifically, by viewing $C$ over $\mathbb{F}_p$, we observe that the condition
\[
\sum_{j = 0}^{n-1}a_j = m \ne 0 \quad\hbox{in $\mathbb{F}_p$}
\]
is guaranteed by the aforementioned range of $m$. Therefore, Theorem \ref{thm:n=ps} directly implies that $C$ is nonsingular over $\mathbb{F}_p$.
\end{example}

\subsection{The Coprime Case 
$\gcd(n,\mathrm{char}(F))=1$}
\label{sec:singularity_gcd}

Let $F=\mathbb{F}_q$ be a finite field, and let $n\ge 1$ satisfy $\gcd(n,\mathrm{char}(\mathbb{F}_q))=1$.
The $n$-th cyclotomic polynomial $\Phi_n(x)\in \mathbb{F}_q[x]$ is the unique monic polynomial whose roots in an algebraic
closure $\overline{\mathbb{F}}_q$ are exactly the primitive $n$-th roots of unity. Equivalently, if $\zeta\in\overline{\mathbb{F}}_q$
is a primitive $n$-th root of unity, then
\[
\Phi_n(x)=\prod_{\substack{1\le s\le n\\ \gcd(n,s)=1}}\bigl(x-\zeta^{\,s}\bigr).
\]
Moreover, one has the factorization
\[
x^n-1=\prod_{d\mid n}\Phi_d(x)
\qquad\text{in }\mathbb{F}_q[x].
\]

\begin{remark}\label{rem:cyclotomic_reduction}
Recall that $\Phi_n(X)$, the classical $n$-th cyclotomic polynomial, belongs to $\mathbb{Z}[X]$. 
When $\gcd(n,\mathrm{char}(\mathbb{F}_q))=1$, the polynomial $\Phi_n(x)$ defined above agrees with the reduction of $\Phi_n(X)$
modulo $\mathrm{char}(\mathbb{F}_q)$ (and hence may be viewed in $\mathbb{F}_q[x]$); see, e.g.,
\cite[Section~1.1]{sanna}.
\end{remark}

\begin{proposition}\label{pro:cyc_field}
Let $C=C(a_0,a_1,\ldots,a_{n-1})$ over $\mathbb{F}_q$ with $\gcd(\mathrm{char}(\mathbb{F}_q), n) = 1$. 
Let 
\(
f(x)=a_0+a_1x+\cdots+a_{n-1}x^{n-1}\in \mathbb{F}_q[x]
\)
be the associated polynomial of $C$. 
Then, $C$ is singular iff
there exists a divisor $d\mid n$ such that
      some irreducible factor over $\mathbb{F}_q[x]$ of the cyclotomic polynomial $\Phi_d(x)$ divides $f(x)$.

\end{proposition}

\begin{pf}
Assume that $C$ is singular. Then, since
\begin{equation}
\label{eq:cyclopoly}
x^n-1 = \prod_{d \mid n} \Phi_d(x),
\end{equation}
we have

\[
\gcd (f(x), x^n-1) = \gcd \bigl(f(x), \prod_{d \mid n} \Phi_d(x)\bigr).
\]
By Lemma~\ref{lem:rank_circulant_any_field}, $\gcd(f(x),x^n-1)\ne 1$ and hence, there exists a divisor $d$ of $n$ such that
$\gcd(f(x), \Phi_d(x)) \neq 1$. Let $h(x)  = \gcd(f(x), \Phi_d(x))$.
Choose an irreducible factor $g(x)$ of $h(x)$; then $g(x)\mid f(x)$ and $g(x)\mid \Phi_d(x)$.

For the other direction, if some irreducible factor $h(x)$ of $\Phi_d(x)$ divides $f(x)$, then $h(\omega)=0$ for some root $\omega$ of $\Phi_d(x)$ and we have $\omega^n=1$ by (\ref{eq:cyclopoly}).
Since $h(x)\mid f(x)$, we also have $f(\omega)=0$.
Let $g(x)$ be the minimal polynomial of $\omega$ over $F$.
Since $\omega^n=1$, the polynomial $x^n-1$ has $\omega$ as a root. Therefore, by Proposition~\ref{prop:min}, the minimal polynomial $g(x)$ divides both $x^n-1$ and $f(x)$, which implies that $\gcd(f(x),x^n-1)\ne 1$.
\hfill $\square$
\end{pf}

Let $\overline{\mathbb{F}}_q$ be an algebraic closure of the finite field $\mathbb{F}_q$.
Fix an integer $n$ with $\gcd(n,q)=1$, and let $\zeta\in\overline{\mathbb{F}}_q^\times$ be a primitive $n$-th root of unity.
Consider the Frobenius automorphism
\[
\sigma:\overline{\mathbb{F}}_q\to\overline{\mathbb{F}}_q,\qquad x\mapsto x^{\,q}.
\]

Let $t\ge 1$ be the smallest integer such that $\sigma^{t}(\zeta)=\zeta$. Then $\zeta\in\mathbb{F}_{q^t}$, and the minimal polynomial of $\zeta$ over $\mathbb{F}_q$ is given by
\[
h_\zeta(x)=\prod_{i=0}^{t-1}\bigl(x-\zeta^{q^i}\bigr)\in\mathbb{F}_q[x].
\]

Note that $\mathrm{ord}_n(q)$ is the least $t\ge1$ such that $\zeta^{q^{t}}=\zeta$.

Since the roots of $\Phi_n(x)$ are precisely the primitive $n$-th roots of unity,
each irreducible factor of $\Phi_n(x)$ is the minimal polynomial over $\mathbb{F}_q$
of one such root.
Since the above degree $t$ depends only on $n$ and $q$ (and not on the choice of $\zeta$), every irreducible factor of $\Phi_n(x)$ over $\mathbb{F}_q$ has degree $t$ as well.
In particular, if $n$ is prime, then $\Phi_n(x)=1+x+\cdots+x^{n-1}$, so $\deg\Phi_n=n-1$.
Hence $\Phi_n(x)$ factors in $\mathbb{F}_q[x]$ as a product of $(n-1)/t$ distinct irreducible polynomials.

In summary, for a prime number $n$ with $\gcd(n,q)=1$, we immediately see that 
$\Phi_n(x)$ is irreducible if and only if
$\mathrm{ord}_n(q)=n-1$.
Now we obtain the following result. 

\begin{proposition}\label{pro:cyclo-irr}
Let $q=p^k$ be a prime power, and let $r$ be a prime with $\gcd(r,p)=1$. Consider
\[
\Phi_r(x)=1+x+\cdots+x^{r-1}\in \mathbb{F}_q[x].
\]
Then $\Phi_r(x)$ is irreducible over $\mathbb{F}_q$ if and only if
$
\mathrm{ord}_r(q)=r-1.
$
\end{proposition}

In the special case where $q=2$ and $r$ is an odd prime, Proposition~\ref{pro:cyclo-irr} recovers the well-known result (see, e.g.,~\cite{ln}) that
\[
1+x+\cdots+x^{r-1}\in \mathbb{F}_2[x]
\]
is irreducible if and only if $2$ is a primitive root modulo $r$ (equivalently, $\mathrm{ord}_r(2)=r-1$).

We assume that $r$ is an odd prime. Under the hypothesis $\mathrm{ord}_r(q)=r-1$, the factorization
\[
x^r-1=(x-1)\Phi_r(x)
\]
has exactly two distinct cyclotomic factors over $\mathbb{F}_q$, and Proposition~\ref{pro:cyclo-irr} guarantees that $\Phi_r(x)$ is irreducible in $\mathbb{F}_q[x]$.

\bigskip
\begin{theorem}\label{thm:prime-size}
Let $r$ be a prime and let $q$ be a prime power with $\gcd(r,q)=1$.
Assume $\mathrm{ord}_r(q)=r-1$. 
Then, $C = C(a_0,a_1,\dots,a_{r-1})$ over $\mathbb{F}_q$ is singular if and only if exactly one of the following three mutually exclusive conditions holds:
\begin{enumerate}
\item[\textnormal{(i)}] $a_0=\cdots=a_{r-1}=0$;
\item[\textnormal{(ii)}] $a_0=\cdots=a_{r-1}=c$ for some nonzero $c\in\mathbb{F}_q$;
\item[\textnormal{(iii)}]
$\sum_{j=0}^{r-1} a_j =0 \ \text{in }\mathbb{F}_q$, with some nonzero $a_j$.
\end{enumerate}
\end{theorem}

\begin{pf}
Let $f(x)=a_0+a_1x+\cdots+a_{r-1}x^{r-1}\in\mathbb{F}_q[x]$
be the associated polynomial of $C$. By Lemma~\ref{lem:rank_circulant_any_field}, $C$ is singular over $\mathbb{F}_q$ if and only if $\gcd(f(x),x^r-1)\neq 1$ in $\mathbb{F}_q[x]$. 
Note that $x^r-1=(x-1)\Phi_r(x)=(x-1)(1+x+\cdots+x^{r-1})$ and that, under the condition $\mathrm{ord}_r(q)=r-1$, $\Phi_r(x)$ is irreducible over $\mathbb{F}_q$ by Proposition~\ref{pro:cyclo-irr}. 
A nonzero polynomial $f(x)$ cannot be divisible by $(x-1)\Phi_r(x)$ since $\deg f \le r-1<\deg((x-1)\Phi_r(x))$. Thus, any nontrivial common divisor of nonzero $f(x)$ and $x^r-1$ must be either $x-1$ or $\Phi_r(x)$.

Suppose that $C$ is singular. If $f(x)=0$, we obtain condition (i) ($a_0=\cdots=a_{r-1}=0$).
Assuming $f(x) \neq 0$, since $\gcd(f(x),x^r-1)\neq 1$, we must have $(x-1) \mid f(x)$ or $\Phi_r(x) \mid f(x)$.
Since $\deg \Phi_r(x) = r-1$, the condition $\Phi_r(x) \mid f(x)$ implies $f(x) = c\,\Phi_r(x)$ for some nonzero $c \in \mathbb{F}_q$.
This gives us two disjoint possibilities for nonzero $f$:
\begin{itemize}
\item $f(x) = c\,\Phi_r(x)$ for some $c \in \mathbb{F}_q \setminus \{0\}$, which is condition (ii);
\item $f(x) = g(x)(x-1)$ for some $g(x) \in \mathbb{F}_q[x]$, which is equivalent to $f(1) = \sum_{j=0}^{r-1} a_j = 0$, with some nonzero $a_j$, yielding condition (iii).
\end{itemize}

Conversely, if condition (i) or (ii) holds, $C$ is clearly singular. If condition (iii) holds, $f(1)=0$, which implies $(x-1) \mid f(x)$, so $\gcd(f(x),x^r-1) \neq 1$. 
\hfill $\square$
\end{pf}

\begin{remark}\label{rem:Artin_short2}
For a fixed integer $q$, a prime $r$ satisfying $\mathrm{ord}_r(q)=r-1$ is called an \emph{Artin prime} for $q$. 
Artin's primitive root conjecture predicts that if $q\neq -1$ and $q$ is not a perfect square, then
there exist infinitely many Artin primes for $q$ (see, e.g.,~\cite{Moree2004}).
In the setting of the preceding theorem, where $q=p^k$, this hypothesis amounts to requiring $k$ to be odd, since for even $k$, one has
$\mathrm{ord}_r(q)\mid (r-1)/2$ for all odd primes $r\nmid q$ (cf.~\cite{ire}).
\end{remark}

\section*{Acknowledgments}
The work was supported by the GIST Research Project grant funded by GIST in 2026 and the National Research Foundation of Korea (NRF) [Grant Number RS-2026-25476719].


\bibliographystyle{abbrv}

\bibliography{cas-refs}



\end{document}